\documentclass[a4paper,11pt]{article}

\usepackage{epsfig}
\usepackage{graphicx}
\usepackage{amsmath}
\usepackage{amssymb}
\usepackage{latexsym}
\usepackage{setspace}
\usepackage{cite}
\usepackage{amsthm}
\usepackage{verbatim}
\usepackage{enumitem}
\usepackage{mathptmx}

\newtheorem{theorem}{Theorem}
\newtheorem{lemma}{Lemma}
\newtheorem{corollary}{Corollary}

\begin{document}
\title{Asymptotics of Queue Length Distributions in Priority Retrial Queues}
\author{Joris Walraevens\footnote{Department of Telecommunications and Information Processing (EA07), Ghent University, Gent, Belgium, Joris.Walraevens@UGent.be}, Dieter Claeys\footnote{Department of Industrial Systems Engineering and Product Design, Ghent University, Gent, Belgium, Dieter.Claeys@UGent.be} and Tuan Phung-Duc\footnote{Division of Policy and Planning Sciences, Faculty of Engineering, Information and Systems, University of Tsukuba, Tsukuba, Japan,  tuan@is.titech.ac.jp}}
\date{}
\maketitle

\begin{abstract}
We calculate asymptotics of the distribution of the number of customers in orbit in a two-class priority retrial $M/G/1$-type queueing model. In this model, priority customers wait in line while non-priority customers join an orbit and retry later. 
Although the generating function and moments of the number of customers in orbit has been analyzed before, asymptotics of the distribution have not been thoroughly investigated. We use singularity analysis of the probability generating function to do just that. Our results show that different regimes exist for these asymptotics in case of light-tailed service times: in what we call the `priority regime', the tail asymptotics have the same decay ($\sim cn^{-3/2}R^{-n}$) as in the priority non-retrial queue and the retrial rate only influences the constant $c$. In the `retrial regime', the retrial rate also influences the sub-exponential factor of the asymptotics. In this regime, asymptotics are very similar to asymptotics in retrial queues without (priority) waiting line. Finally, we also analyze the case that the service time distribution is power law (with or without exponential cut-off) using the same technique.
\end{abstract}

\section{Introduction}

Retrial queues are featured by the fact that arriving customers that cannot occupy the server upon arrival join a virtual waiting room (called orbit) and retry for service after some random time. Retrial phenomena naturally occur in various applications such as service systems e.g., call centers~\cite{Aguir04} and telecommunication systems, e.g., cellular networks~\cite{tran-gia97}. It is important to take retrial behavior into account in modeling of these systems. However, the flow of retrial customers makes the arrival process to the server complex and as a result the analysis of retrial queues is more difficult than that of the corresponding conventional models with an infinite buffer. 

On the other hand, priority is essential in service systems where service-differentiation is required. In fact, customers to service systems have different needs and restrictions. Some customers have urgent demands while relatively long waiting time is acceptable for other type of customers. For example, in hospitals, urgent patients should be treated as soon as possible while others can wait longer. So, queueing systems with priority are very important from a practical point of view. In the basic single server priority queue~\cite{Walraevens03}, there are two queues for two classes of customers. Upon the service completion of a customer, a customer of a lower priority class is served if and only if there is no customer of the high priority class present. 

The basic priority model is extended to a retrial setting in~\cite{Choi_Park90,Falin_Artalejo_Martin93,Dimitriou13, Kim_Kim17}. In particular, priority customers who find the server busy upon arrival can wait at the priority queue in front of the server while blocked lower priority customers join the orbit and retry to enter the server after some random time. In fact, such a priority retrial queue comes naturally if customers have the choice to either wait in line or retry later. They are also a perfect fit to model cellular networks. In cellular networks, the cover area is divided into cells and each cell is served by one base station. The base station serves both fresh and handover calls. Fresh calls are originated from the cell of the base station while handover calls arrive from adjacent cells. Since a handover call has been connecting with the base station of an adjacent cell, it should be served at the targeting cell as soon as possible and thus should be given priority. Furthermore, adjacent cells in cellular networks are overlapping and thus, in the overlapping area, a handover call can connect to both original and targeting base stations. When the handover call is in a overlap area, it could be considered that it is in the priority queue in our model. On the other hand, fresh calls are originated from the cell of the base station. So, when the base station is busy, fresh calls may be blocked and they may retry at a later time. Thus the orbit for our model are the pool for blocked fresh calls while the buffer is the overlapping  area in cellular networks for handover calls.

In the work on priority retrial queues mentioned above, the probability generating function (PGF) of the number of customers in orbit is obtained, albeit in a complicated form. As a result, calculated performance measures are restricted to moments and it is difficult to gain insight into the structure of their distributions. This motivated us to investigate the asymptotics of the distribution of the number of customers in orbit starting from its PGF. 

Queue length asymptotics for retrial queues have been paid much attention to in the literature. Queue length asymptotics for the single class single server retrial queue (without an infinite buffer in front of the server) have been investigated in~\cite{Kim_Ko07,Yamamuro12,Shang06}. In~\cite{Kim_Ko07}, asymptotics of the queue length of the M/G/1 retrial queue are presented under light-tailed assumption of the service time distribution, while in~\cite{Shang06} and \cite{Yamamuro12} tail asymptotics of the queue length are derived under heavy tailed assumption of the service time distribution for the M/G/1 and M$^X$/G/1 retrial queues, respectively. In~\cite{Artalejo_Phung13}, tail asymptotics for the queue length are obtained for the M/G/1 retrial queue with two-way communication where the server makes outgoing calls in its idle time. In~\cite{Artalejo_Phung13}, the service time distributions of incoming calls and outgoing calls follow an arbitrary distribution. 

Asymptotics of queue length and waiting time distributions have also been investigated thoroughly for priority queues. It has been established that queue length and waiting time distributions in priority queues with light-tailed service times follow one of three laws: exponential $\sim R^{-n}$, power law with exponential cut-off ($\sim n^{-3/2}R^{-n}$), or a border case ($\sim n^{-1/2}R^{-n}$), cf.\ \cite{Abate97,Walraevens03,Laevens98}. This is later extended to asymptotics of the joint distribution in the preemptive and non-preemptive $M/M/1$ priority queues in \cite{Li_Zhao09,Li_Zhao11b} and to asymptotics of the total system content in the corresponding discrete-time non-preemptive queue with batch arrivals \cite{Maertens07}.

To the best of our knowledge, asymptotics of the number of customers in orbit in a priority retrial queue has not been investigated in the literature. It should be noted that the underlying stochastic process of our priority retrial queue is two-dimensional where the first component is the number of customers in the priority queue and the second component keeps track of the number of customers in the orbit. Therefore, our model and results are connected to asymptotic analysis of two-dimensional Markov chains~\cite{Li_Zhao11b,Li_Zhao11a,Song15,Kobayashi_Miyazawa14,He_Li_Zhao09,Haque_Zhao_Liu05,Kobayashi10}. 
As will also be apparent from our results, these studies on two-dimensional random walks in the first quadrant show that distributions of queue lengths are not necessarily asymptotically exponential. However, the service times in our model follow an arbitrary distribution, including heavy-tailed distributions, and therefore the underlying stochastic process of our model does not fall into the framework of two-dimensional Markov chains. Furthermore, since we consider the classical retrial policy in which the retrial rate is proportional to the number of customers in the orbit, the transition structure of our stochastic process is non-homogeneous which is also not the case in the work mentioned above. 

In this paper, we investigate the singularities of the PGFs derived in~\cite{Choi_Chang99,Choi_Park90}, and the behavior of these PGFs in the neighborhood of these singularities to obtain asymptotics for the number of customers in the orbit (when the server is free, when the server is busy, and unconditionally). We detail the model and write down the expressions of the PGFs we start from in the next section. In Section \ref{approach}, we convey our general approach which exists of analyzing the behavior of the PGF in the neighborhood of its dominant singularity (singularity with smallest norm) and using standard transfer formulas to calculate the asymptotics of the corresponding distributions. In Section \ref{exponential}, we demonstrate this approach in great detail to obtain asymptotics when service times are light-tailed (more precisely if the Laplace-Stieltjes Transform (LST) $B^*(s)$ is infinite in its dominant singularity). The main complication here is finding the dominant singularity and the correct asymptotic formula of the PGF in the neighborhood of this singularity. This is due to the complex expression of the PGF. We show that we obtain two regimes for the asymptotics (and a third border regime). In the first regime, the asymptotics have the same structure as that in the conventional priority queue while in the second regime the retrial rate plays an important role. Finally, in Section \ref{power_law}, we consider the case that the singularity of the LST of the service times plays a more crucial role. This is the case when service times are power law, but also in a specific case where the service times are power law with an exponential cut-off. We fully investigate the asymptotics of the number of customers in orbit, using the same technique as in Section 4. This seems to be novel as alternative techniques (Tauberian theorems, large deviations) are usually adopted for the latter case \cite{Zwart01}.

\section{Model and Preliminary Results}

We assume an M/G/1 retrial queue with two types of customers. Type-I customers arrive in the system according to a Poisson process with rate $\lambda_1$, occupy the server when it is free or join an (infinitely-sized) regular queue when the server is busy. Type-II customers arrive in the system according to a Poisson process with rate $\lambda_2$, occupy the server when it is free, or join an (infinitely-sized) orbit when it is busy. Service times are generally distributed with distribution $B(\cdot)$, density $b(x)$, Laplace transform $B^*(s)$ and mean $b$. When a service ends, the served customer departs the system and a new customer from the (regular) queue (if any) enters the server. Within the regular queue the service order is First-Come-First-Served (FCFS). The customers in the orbit retry to access the server according to independent Poisson processes with rate $\nu$, and succeed only when the server is available at the moment of retrial. It is clear from above description that type-I customers have strict non-preemptive priority over type-II customers, since the latter can only access the server when no type-I customers are present. 

This model has been analyzed extensively in \cite{Choi_Park90,Choi_Chang99}. Theorem 4.1 in \cite{Choi_Chang99} contains expressions for the partial PGFs of the number of customers in both queues (not counting the one in service) when the server is either busy or idle. Note that the number of customers in the regular queue is zero when the server is idle. The formulas read:
\begin{align}
Q(z)={} & (1-\rho)\exp\left[-\dfrac{\lambda}{\nu}\int_{z}^1\dfrac{1-h(x)}{h(x)-x}dx\right], \label{Q} \\
P^*(0,z_1,z_2)={} & (1-\rho)\exp\left[-\dfrac{\lambda}{\nu}\int_{z_2}^1\dfrac{1-h(x)}{h(x)-x}dx\right]\dfrac{\lambda-\lambda B^*(\lambda_1(1-z_1)+\lambda_2(1-z_2))}{\lambda_1(1-z_1)+\lambda_2(1-z_2)} \nonumber \\
& \cdot\dfrac{z_1-h(z_2)}{z_1-B^*(\lambda_1(1-z_1)+\lambda_2(1-z_2))}\dfrac{1-z_2}{h(z_2)-z_2}, \nonumber
\end{align}
with $Q(z)$ the partial PGF of the number of customers in orbit when the server is idle and $P^*(0,z_1,z_2)$ the partial joint PGF of the number of customers in queue and in orbit when the server is busy, $\lambda:=\lambda_1+\lambda_2$ the total arrival rate, $\rho:=\lambda b$ the total arrival load and $h(z_2)$ the unique root in $z_1$ with $|z_1|<1$ of the equation $z_1-B^*(\lambda_1(1-z_1)+\lambda_2(1-z_2))=0$ for given $z_2$ with $|z_2|<1$.

We are, in particular, interested in the asymptotics of the number of type-II customers present in the orbit. We will asymptotically invert three PGFs, namely, (i) the partial PGF $Q(z)$ of the number of type-II customers in orbit when the server is idle, given by equation \eqref{Q}, (ii) the partial PGF of the number of type-II customers in orbit when the server is busy, expressed by
\begin{align}
R(z):={} & P^*(0,1,z) \nonumber \\
={} & (1-\rho)\exp\left[-\dfrac{\lambda}{\nu}\int_{z}^1\dfrac{1-h(x)}{h(x)-x}dx\right]\dfrac{\lambda}{\lambda_2}\dfrac{1-h(z)}{h(z)-z}, \label{R}
\end{align}
and, (iii) the PGF of the number of type-II customers in orbit
\begin{align}
P_2(z):={} & Q(z)+R(z) \nonumber \\
={} & (1-\rho)\exp\left[-\dfrac{\lambda}{\nu}\int_{z}^1\dfrac{1-h(x)}{h(x)-x}dx\right]\left\{1+\dfrac{\lambda}{\lambda_2}\dfrac{1-h(z)}{h(z)-z}\right\}. \label{P2}
\end{align}

\section{Problem Statement and Approach} \label{approach}

We are interested in the asymptotics of the (partial) distribution of the stationary number of class-II customers in orbit. All information of a discrete distribution is, in theory, incorporated in its PGF. In this section, we will disclose our general approach to calculate asymptotics of a distribution from its PGF. 

Assume a generic (possibly defective) discrete distribution $\{f_n\}_{n=0}^{\infty}$ with PGF $f(z)$. The $f_n$ are probabilities and thus satisfy two conditions: $f_n\geq0$ and $\sum_{n=0}^{\infty}f_n\leq1$. The PGF $f(z)$ is an analytic function at $0$ with Taylor expansion
\begin{align}
f(z)=\sum_{n=0}^{\infty}f_nz^n, \label{PGF}
\end{align}
i.e., the coefficients of its Taylor expansion correspond to the distribution.

Our intention is the asymptotic calculation of the distribution $\{f_n\}_{n=0}^{\infty}$ for $n\rightarrow\infty$. This is identical to the asymptotic calculation of the coefficients $\{[z^n]f(z)\}_{n=0}^{\infty}$ of the Taylor expansion of its PGF $f(z)$, cf.\ \eqref{PGF}.\footnote{$[z^n]f(z)$ denotes the $n$-th coefficient of the Taylor expansion of $f(z)$.} The latter can be done through singularity analysis of that function $f(z)$. Flajolet and Sedgewick \cite{Flajolet08} study this problem comprehensively. They state that asymptotics can generally be written as 
\begin{align}
[z^n]f(z)\sim{} R^{-n}\theta(n), \label{inversion}
\end{align}
for $n\rightarrow\infty$, corresponding to an exponential factor $R^{-n}$ modulated by a subexponential tame factor $\theta(n)$, where we write $f_n\sim g_n$ if $n\rightarrow\infty$ if $\lim_{n\rightarrow\infty}g_n/f_n=1$. These two factors in \eqref{inversion} correspond to what Flajolet and Sedgewick call the two principles of coefficient asymptotics \cite{Flajolet08}, p.\ 227: 
\begin{enumerate}
\item The \textit{location} of a function's 
singularities dictates the exponential growth ($R^{-n}$) of its coefficients. 
\item The \textit{nature} of a function's singularities determines the associate subexponential factor $(\theta(n))$.
\end{enumerate}

The factor $R^{-n}$ results from the easy rescaling rule
\begin{align}
[z^n]f(z)={} & R^{-n}[z^n]f(Rz). \label{rescale}
\end{align}
Choosing $R$ as a singularity of $f(z)$, the normalized function $f(Rz)$ is singular at $1$. The second factor arises from singularity analysis of the normalized function in $1$. This factor turns out to be sub-exponential due to \cite{Flajolet08}, corollary VI.1, p.\ 392. This corollary states that for a function $f(z)$ that is analytic in a so-called $\Delta$-domain at $1$ the following property holds:
\begin{align}
f(z)\sim(1-z)^{-\alpha_f} \mbox{ as } z\rightarrow1, \, z\in\Delta \qquad \Rightarrow \qquad [z^n]f(z)\sim{} \dfrac{n^{\alpha_f-1}}{\Gamma(\alpha_f)} \mbox{ as } n\rightarrow\infty, \label{coefficient}
\end{align} 
where we write $f(z)\sim g(z)$ as $z\rightarrow1, z\in \Delta$ if the ratio $g(z)/f(z)$ tends to $1$ as $z\rightarrow1$ in $\Delta$. A $\Delta$-domain at $1$ is a domain $\Delta=\{z \big{|} \, |z|<S, z\neq1,|\arg(z-1)|>\phi\}$ for some $S>1$ and $0<\phi<\pi/2$. In the remainder, we will write $z\rightarrow 1$ and implicitly assume the definition of a $\Delta$-domain at $1$ and $z\in\Delta$.\footnote{We can extend this to a $\Delta$-domain to any complex number $\zeta$ by the mapping $z\rightarrow \zeta z$.} 

Each singularity of $f(z)$ yields a term in the coefficients as \eqref{coefficient}, since contributions of multiple singularities are to be added up (\cite{Flajolet08}, section VI.5, p.\ 398). However, due to the exponential factor $R^{-n}$ the ones with smallest norm dominate (hence the name \textit{dominant singularities}). For a function $f$ that has a Taylor series with non-negative coefficients (which is the case for PGFs), Pringsheim's theorem\footnote{Also known as Vivanti's theorem or the Vivanti-Pringsheim theorem.} (\cite{Flajolet08}, Theorem VI.6. p.\ 240) states that one of the dominant singularities is the real positive point $z=R_f$, with $R_f$ the radius of convergence of $f$. Furthermore, $R_f\geq1$ since $\sum_{n=0}^{\infty}f_n$ converges. In principle, other singularities may lie on the radius of convergence, but these lead to some periodicity of the coefficients, cf.\ \cite{Flajolet08}, section IV.6.1, p.\ 263. Since buffer occupancy distributions do usually not demonstrate periodicity, we will assume aperiodicity of our sequences, or, equivalently, it is assumed that only $R_f$ is a dominant singularity of the corresponding PGF $f(z)$. 

Summarized, we have that
\begin{align}
f(z)\sim c_f\left(1-\dfrac{z}{R_f}\right)^{-\alpha_f} \mbox{ as } z\rightarrow R_f \qquad \Rightarrow \qquad [z^n]f(z)\sim{} c_f\dfrac{n^{\alpha_f-1}}{\Gamma(\alpha_f)}R_f^{-n} \mbox{ as } n\rightarrow\infty, \nonumber
\end{align} 
or, equivalently,
\begin{align}
f(z)\sim c_f\left(R_f-z\right)^{-\alpha_f} \mbox{ as } z\rightarrow R_f \qquad \Rightarrow \qquad [z^n]f(z)\sim{} c_f\dfrac{n^{\alpha_f-1}}{\Gamma(\alpha_f)}R_f^{-n-\alpha_f} \mbox{ as } n\rightarrow\infty. \label{asymptotic_inversion}
\end{align} 

Our approach can be summarized as: we write expressions \eqref{Q}-\eqref{P2} as in the left side of \eqref{asymptotic_inversion} to find the asymptotics of the corresponding distribution as in the right side of \eqref{asymptotic_inversion}.

As in \cite{Abate97}, we identify three types of service time distributions, based on the dominant singularity $-R_{B^*}$ of its LST $B^*(s)$ and on the behavior of the LST in the neighborhood of that singularity\footnote{We assume $R_{B^*}=\infty$ if $B^*(s)$ is analytic in the complex plane.}, namely
\begin{enumerate}
\item Type-1: $R_{B^*}>0$ and $B^*(-R_{B^*})=\infty$,
\item Type-2: $R_{B^*}>0$ and $1<B^*(-R_{B^*})<\infty$,
\item Type-3: $R_{B^*}=0$ and $B^*(-R_{B^*})=1$.
\end{enumerate}
Distributions whose LST goes to $\infty$ for $s\rightarrow-R_{B^*}$ are type-1 distributions. These include asymptotically exponential distributions ($-{R_{B^*}}$ is a pole of $B^*(s)$) but also some distributions that are power law with exponential cut-off ($B^*(s)\sim c_{B^*}(s+R_{B^*})^{-\alpha_{B^*}}$ for $z\rightarrow -R_{B^*}$ with $\alpha_{B^*}>0$ and not an integer). Power-law distributions with exponential cut-off are distributions that are power law for smaller values, but that smoothly turns into an exponentially declining function for larger values \cite{Milojevic10}. The other part of these distributions are of type 2 ($B^*(s)\sim B^*(-R_{B^*})+c_{B^*}(s+R_{B^*})^{-\alpha_{B^*}}$ for $z\rightarrow -R_{B^*}$ with $\alpha_{B^*}<0$ and not an integer). Finally, power-law distributions are examples of type-3 distributions, as the dominant singularity of their LST is $0$ ($B^*(s)\sim 1+c_{B^*}s^{-\alpha_{B^*}}$ for $z\rightarrow 0$ with $\alpha_{B^*}<0$ and not an integer).

\section{Type 1 Service Time Distributions} \label{exponential}

We use singularity analysis of PGFs to calculate the asymptotics of the corresponding distributions. We therefore calculate asymptotic expressions of \eqref{Q}-\eqref{P2} in the neighborhood of their respective dominant singularities, which are also their radii of convergence. This is the main contribution of this article. Once this suitable form is found, it is a mere application of singularity analysis (cf.\ \eqref{asymptotic_inversion}) to find the asymptotics of the corresponding distributions. 

In this section, we assume the service times to be of type 1, also called `well-behaved' functions in \cite{Abate97}. We analyze some cases where the service times are of type 2 or type 3 in the next section. Throughout the general analysis, we assume $\lambda_1,\lambda_2>0$. The case $\lambda_1=0$ is treated later as a special case.

\subsection{Asymptotic Inversion of $Q(z)$}

We start with the asymptotic inversion of expression \eqref{Q}. To this end, we write $Q(z)$ in the suitable form $c_Q\cdot(R_Q-z)^
{-\alpha_Q}$ as explained in section \ref{approach}. We do this step by step by first studying asymptotic behavior of functions $Q$ is composed of in the neighborhood of their dominant singularity. In some functions, we will also need information on the error term, which will be denoted by $O(\cdot)$, with $f=O(g)$ if $\limsup |f|/|g|<\infty$.

We commence with singularity analysis of the implicitly defined function $h(z)$.
\begin{lemma}
The dominant singularity $R_h$ of the implicit function $h(z)$, defined as the unique root in $w$ with $|w|<1$ of the equation $w=B^*(\lambda_1(1-w)+\lambda_2(1-z))$ for $|z|<1$, is the smallest positive real solution of 
\begin{align}
-\lambda_1{B^*}'(\lambda_1(1-h(z))+\lambda_2(1-z))=1. \label{deriv_functional}
\end{align}
The function $h(z)$ is equal to
\begin{align}
h(z)={} & h(R_h)-c_h(R_h-z)^{1/2}+O(R_h-z), \label{asymptotic_h}
\end{align}
as $z\rightarrow R_h$. The constant $c_h$ is given by
\begin{align}
c_h={} & \sqrt{\dfrac{2\lambda_2}{\lambda_1^3{B^*}''(\lambda_1(1-h(R_h))+\lambda_2(1-R_h))}}. \label{c_h}
\end{align}
\label{lemma_h}
\end{lemma}
\begin{proof}
The function $h(z)$ is implicitly defined as
\begin{align}
h(z)={} & G(z,h(z)), \nonumber
\end{align}
with 
\begin{align}
G(z,w)={} & B^*(\lambda_1(1-w)+\lambda_2(1-z)). \label{G}
\end{align}
It is a variation of the so-called smooth implicit-function schema (\cite{Flajolet08}, Definition VII.4, p.\ 467) and $h(z)$ contains therefore a square-root singularity in the point $R_h$, with $(R_h,h(R_h))$ the solution of the characteristic system in $(z,w)$:
\begin{align}
G(z,w)={} & w, \label{characteristic1} \\
G_w(z,w)={} & 1, \label{characteristic2} 
\end{align}
with $G_w(z,w)$ the partial derivative of $G$ in $w$. The first equation is the implicit function definition of $h(z)$; the second equation leads to \eqref{deriv_functional}. We discuss the smooth implicit-function schema and the existence of such a square-root singularity of an implicit function in the Appendix.

According to \cite{Flajolet08}, Theorem VII.3, p.\ 468, $h(z)$ converges at $z=R_h$ as written down in \eqref{asymptotic_h} with
\begin{align}
c_h={} & \sqrt{\dfrac{2G_z(R_h,h(R_h))}{G_{ww}(R_h,h(R_h))}}. \nonumber
\end{align}
After some algebra and use of \eqref{deriv_functional}, this latter equation leads to expression \eqref{c_h}.
\end{proof}

Next, we develop singularity analysis of the integrand of the integral in \eqref{Q}. Different cases are observed.
\begin{lemma}
The dominant singularity  $R_{f_1}$ of
\begin{align}
f_1(z) := \dfrac{1-h(z)}{h(z)-z} \label{f1}
\end{align}
is equal to the smallest positive zero $R^*$ larger than $1$ of $h(z)-z$, if any, or $R_h$, if not. The asymptotic behavior of $f_1(z)$ near that dominant singularity is one of three cases:
\begin{align}
f_1(z)={} & \left\{\begin{array}{cc}
\dfrac{c^{(1)}_{f_1}}{R_{f_1}-z}+O(1) & \mbox{ if } R_{f_1}=R^*<R_h \\
f_1(R_{f_1})-c^{(2)}_{f_1}(R_{f_1}-z)^{1/2}+O(R_{f_1}-z) & \mbox{ if } R_{f_1}=R_h \mbox{ and } R^* \mbox{ does not exist} \\
\dfrac{c^{(3)}_{f_1}}{(R_{f_1}-z)^{1/2}}+O(1) & \mbox{ if } R_{f_1}=R^*=R_h \\
\end{array}\right., \nonumber
\end{align}  
as $z\rightarrow R_{f_1}$. The constants in the above formula are given by
\begin{align}
c^{(1)}_{f_1}={} &  \dfrac{(1-R_{f_1})(1+\lambda_1{B^*}'(\lambda(1-R_{f_1})))}{1+\lambda{B^*}'(\lambda(1-R_{f_1}))}, \nonumber \\
c^{(2)}_{f_1}={} & \dfrac{R_{f_1}-1}{(h(R_{f_1})-R_{f_1})^2}c_h, \nonumber \\
c^{(3)}_{f_1}={} & \dfrac{R_{f_1}-1}{c_h}. \nonumber
\end{align} 
\label{lemma_f_1}
\end{lemma}

\begin{proof}
The singularity $R_h$ of $h(z)$ carries over to $f_1(z)$. A second potentially dominant singularity of $f_1(z)$ is the smallest positive zero larger than $1$ of its denominator. Say that zero is denoted by $R^*$, if it exists. Note that $h(z)$ is analytic for $|z|<R_h$ and, therefore $R^*\leq R_h$ if $R^*$ exists. We have three cases, depending on the relative positions of $R_h$ and $R^*$ and whether the latter exists:
\begin{enumerate}
\item $R^*$ exists and is smaller than $R_h$: in this case, $R_{f_1}=R^*$ and this is a pole with multiplicity $1$. The multiplicity is due to $h'(R^*)>1$, which is a result of $h(z)$ being a convex function, $h(1)=1$, $h'(1)=\rho_2/(1-\rho_1)<1$ and $R^*>1$. This leads to the first case in the lemma.
\item $R^*$ does not exist: in this case, $R_{f_1}=R_h$ and we can use equation \eqref{asymptotic_h} to write
\begin{align}
f_1(z)\sim{} & \dfrac{1-h(R_{f_1})+c_h(R_{f_1}-z)^{1/2}+O(R_{f_1}-z)}{h(R_{f_1})-c_h(R_{f_1}-z)^{1/2}+O(R_{f_1}-z)-R_{f_1}} \nonumber \\
={} & (1-h(R_{f_1})+c_h(R_{f_1}-z)^{1/2}+O(R_{f_1}-z))\left(\dfrac{1}{h(R_{f_1})-R_{f_1}}+\dfrac{c_h(R_{f_1}-z)^{1/2}}{(h(R_{f_1})-R_{f_1})^2}+O(R_{f_1}-z)\right) \nonumber \\
={} & f_1(R_{f_1})-c^{(2)}_{f_1}(R_{f_1}-z)^{1/2}+O(R_{f_1}-z), \nonumber
\end{align}
with the constant $c_{f_1}^{(2)}$ as in the lemma. We expanded the denominator in the second step and kept the first two terms, while we kept the dominating terms of the expansion of the resulting product in the third step.
\item $R^*$ exists and is equal to $R_h$: this is a boundary case of the other two cases. We can use equation \eqref{asymptotic_h} and $h(R_{f_1})=R_{f_1}$ to write
\begin{align}
f_1(z)={} & \dfrac{1-R_{f_1}+c_h(R_{f_1}-z)^{1/2}+O(R_{f_1}-z)}{R_{f_1}-c_h(R_{f_1}-z)^{1/2}+O(R_{f_1}-z)-z} \nonumber \\
={} & \dfrac{R_{f_1}-1}{c_h(R_{f_1}-z)^{1/2}}+O(1), \nonumber
\end{align}
which equals the third case of the lemma.
\end{enumerate}
\end{proof}

The next function we analyze asymptotically in the neighborhood of its singularity is the integral in \eqref{Q}.
\begin{lemma}
The dominant singularity $R_{f_2}$ of
\begin{align}
f_2(z) := -\dfrac{\lambda}{\nu}\int_{z}^1\dfrac{1-h(x)}{h(x)-x}dx \nonumber
\end{align}
is equal to $R_{f_1}$. The asymptotic behavior of $f_2(z)$ near that dominant singularity is one of three cases:
\begin{align}
f_2(z)={} & \left\{\begin{array}{llc}
 -\dfrac{\lambda}{\nu}c^{(1)}_{f_1} & \ln(R_{f_2}-z)+\dfrac{\lambda}{\nu}c^{(1)}_{f_1}\ln(R_{f_2}-1) & \\
& -\dfrac{\lambda}{\nu}\int_{R_{f_2}}^1\left(f_1(x)-\dfrac{c^{(1)}_{f_1}}{R_{f_2}-x)}\right)dx+O(R_{f_2}-z) & \mbox{ if } R_{f_2}=R^*<R_h \\
f_2(R_{f_2}) & -\dfrac{\lambda}{\nu}f_1(R_{f_2})(R_{f_2}-z)+\dfrac{2\lambda}{3\nu}c^{(2)}_{f_1}(R_{f_2}-z)^{3/2} & \\
& +O((R_{f_2}-z)^2) & \mbox{ if } R_{f_2}=R_h \mbox{ and } R^* \mbox{ does not exist} \\
f_2(R_{f_2}) & -\dfrac{2\lambda}{\nu}c^{(3)}_{f_1}(R_{f_2}-z)^{1/2}+O(R_{f_2}-z) & \mbox{ if } R_{f_2}=R^*=R_h \\
\end{array}\right., \label{asymptotics_f2(z)}
\end{align}
with the constants as in Lemma \ref{lemma_f_1}.
\label{lemma_f_2}
\end{lemma}

\begin{proof}
According to \cite{Flajolet08}, section VI.10.1, p.\ 418, integration of a function that is amenable to singularity analysis preserves that property. Basically, the asymptotic expansion of the integration equals the integration of the asymptotic expansion of the integrand. According to \cite{Flajolet08}, Theorem VI.9, p.\ 420, $R_{f_2}=R_{f_1}$, so we analyse the three cases of Lemma \ref{lemma_f_1} separately.
\begin{enumerate}
\item $R^*$ exists and is smaller than $R_h$: we can use \cite{Flajolet08}, Theorem VI.9, p.\ 420 and remark VI.28, p.\ 422 to find
\begin{align}
f_2(z)={} & -\dfrac{\lambda}{\nu}\int_{z}^1\left(\dfrac{c^{(1)}_{f_1}}{R_{f_2}-x}+\left(f_1(x)-\dfrac{c^{(1)}_{f_1}}{R_{f_2}-x}\right)\right)dx \nonumber \\
={} & -\dfrac{\lambda}{\nu}c^{(1)}_{f_1}\ln(R_{f_2}-z)+\dfrac{\lambda}{\nu}c^{(1)}_{f_1}\ln(R_{f_2}-1)-\dfrac{\lambda}{\nu}\int_{R_{f_2}}^1\left(f_1(x)-\dfrac{c^{(1)}_{f_1}}{R_{f_2}-x}\right)dx \nonumber \\
& +\dfrac{\lambda}{\nu}\int_z^{R_{f_2}}\left(f_1(x)-\dfrac{c^{(1)}_{f_1}}{R_{f_2}-x)}\right)dx, \nonumber
\end{align}
and the last term is $O((R_{f_1}-z))$ since the integrand is $O(1)$, see Lemma \ref{lemma_f_1}.
\item $R^*$ does not exist: we can use \cite{Flajolet08}, Theorem VI.9, p.\ 420 to find
\begin{align}
f_2(z)={} & f_2(R_{f_2})-\dfrac{\lambda}{\nu}\int_{z}^{R_{f_2}}f_1(x)dx \nonumber \\
={} & f_2(R_{f_2})-\dfrac{\lambda}{\nu}\int_{z}^{R_{f_2}}\left(f_1(R_{f_2})-c^{(2)}_{f_1}(R_{f_2}-x)^{1/2}+O(R_{f_2}-x)\right)dx \nonumber \\
={} & f_2(R_{f_2})-\dfrac{\lambda}{\nu}f_1(R_{f_2})(R_{f_2}-z)+\dfrac{2\lambda}{3\nu}c^{(2)}_{f_1}(R_{f_2}-z)^{3/2}+O((R_{f_2}-z)^2) \nonumber
\end{align}
\item $R^*$ exists and is equal to $R_h$: we use \cite{Flajolet08}, Theorem VI.9, p.\ 420 to find
\begin{align}
f_2(z)={} & f_2(R_{f_2})-\dfrac{\lambda}{\nu}\int_{z}^{R_{f_2}}f_1(x)dx \nonumber \\
={} & f_2(R_{f_2})-\dfrac{\lambda}{\nu}\int_{z}^{R_{f_2}}\left(\dfrac{c^{(3)}_{f_1}}{(R_{f_2}-z)^{1/2}}+O(1)\right)dx \nonumber \\
={} & f_2(R_{f_2})-\dfrac{2\lambda}{\nu}c^{(3)}_{f_1}(R_{f_2}-z)^{1/2}+O(R_f-z) \nonumber
\end{align}
\end{enumerate}

\end{proof}

Finally, we are ready to do singularity analysis of $Q(z)$ and to asymptotically invert it.

\begin{theorem}
The dominant singularity $R_{Q}$ of $Q(z)$ equals $R_{f_1}$. The asymptotic behavior of $Q(z)$ near that dominant singularity is one of three cases:
\begin{align}
Q(z)\sim{} & \left\{\begin{array}{cc}
c_Q^{(1)}(R_Q-z)^{-\dfrac{\lambda}{\nu}c^{(1)}_{f_1}} & \mbox{ if } R_Q=R^*<R_h \\
Q(R_Q)-(1-\rho)\exp\left(f_2(R_Q)\right)\dfrac{\lambda}{\nu}f_1(R_Q)(R_Q-z) &  \\
+c_Q^{(2)}(R_Q-z)^{3/2} & \mbox{ if } R_Q=R_h \mbox{ and } R^* \mbox{ does not exist} \\
Q(R_Q)-c_Q^{(3)}(R_Q-z)^{1/2} & \mbox{ if } R_Q=R^*=R_h \\
\end{array}\right., \label{asymptotics_Q}
\end{align}
with 
\begin{align}
c_Q^{(1)}={} & (1-\rho)(R_Q-1)^{\dfrac{\lambda}{\nu}c^{(1)}_{f_1}}\exp\left(-\dfrac{\lambda}{\nu}\int_{R_Q}^1\left(f_1(x)-\dfrac{c^{(1)}_{f_1}}{R_Q-x}\right)dx\right), \nonumber \\
c_Q^{(2)}={} & (1-\rho)\exp\left(f_2(R_Q)\right)\dfrac{2\lambda}{3\nu}c^{(2)}_{f_1}, \nonumber \\
c_Q^{(3)}={} & (1-\rho)\exp\left(f_2(R_Q)\right)\dfrac{2\lambda}{\nu}c^{(3)}_{f_1}, \nonumber
\end{align}
and the constants $c_{f_1}^{(i)}$ as in Lemma \ref{lemma_f_1}.
\label{theorem_asymptotics_Q}
\end{theorem}

\begin{proof}
The function $Q(z)$ is the composition of $f_3$ and $f_2$ with $f_2$ given in Lemma \ref{lemma_f_2} and 
\begin{align}
f_3(z):={} & (1-\rho)\exp\left(z\right). \nonumber
\end{align}
The asymptotic behavior of compositions is treated in \cite{Flajolet08}, section VI.9, p.\ 411. In general, three different cases arise depending on the value of $f_2(R_{f_2})$ in comparison with $R_{f_3}$. Since, in the current case, $f_3(z)$ is analytic in the whole complex plane, $R_{f_3}=\infty$. Therefore, two of the three possible cases can occur, depending on whether  $f_2(R_{f_2})$ is finite. We treat the three cases for the asymptotic behavior of $f_2(z)$ established in Lemma \ref{lemma_f_2} separately.

\begin{enumerate}
\item $R_{f_2}=R^*<R_h$: $R_{f_2}$ is a logarithmic singularity of $f_2(z)$ and thus $f_2(R_{f_2})=\infty$. This case is therefore of the critical type ($R_{f_3}=f_2(R_{f_2})$). According to \cite{Flajolet08}, section VI.9, p.\ 411, the singularity type is then a mix of the types of both functions $f_2$ and $f_3$. Function $f_2(z)$ is of $\left(-\dfrac{\lambda}{\nu}c^{(1)}_{f_1},0\right)$-logarithmic type (\cite{Flajolet08}, definition VII.1, p.\ 446), while $f_3(z)$ is of exponential type. We therefore have an exp-log schema (\cite{Flajolet08}, section VII.2, p.\ 446), $R_Q=R_{f_2}$ and
\begin{align}
Q(z)={} & (1-\rho)(R_Q-z)^{-\dfrac{\lambda}{\nu}c^{(1)}_{f_1}}(R_Q-1)^{\dfrac{\lambda}{\nu}c^{(1)}_{f_1}}\exp\left(-\dfrac{\lambda}{\nu}\int_{R_Q}^1\left(f_1(x)-\dfrac{c^{(1)}_{f_1}}{R_Q-x)}\right)dx\right)\exp\left(O(R_Q-z)\right) \nonumber \\
={} & (1-\rho)(R_Q-1)^{\dfrac{\lambda}{\nu}c^{(1)}_{f_1}}\exp\left(-\dfrac{\lambda}{\nu}\int_{R_Q}^1\left(f_1(x)-\dfrac{c^{(1)}_{f_1}}{R_Q-x)}\right)dx\right)(R_Q-z)^{-\dfrac{\lambda}{\nu}c^{(1)}_{f_1}}\cdot(1+O(R_Q-z)) \nonumber \\
\sim{} & (1-\rho)(R_Q-1)^{\dfrac{\lambda}{\nu}c^{(1)}_{f_1}}\exp\left(-\dfrac{\lambda}{\nu}\int_{R_Q}^1\left(f_1(x)-\dfrac{c^{(1)}_{f_1}}{R_Q-x)}\right)dx\right)(R_Q-z)^{-\dfrac{\lambda}{\nu}c^{(1)}_{f_1}}, \nonumber
\end{align}
as $z\rightarrow R_Q$. 
\item $R_{f_2}=R_h$ and $R^*$ does not exist: $f_2(R_{f_2})<\infty$ and, therefore, this case is of the subcritical type ($f_2(R_{f_2})<R_{f_3}$). In this case, the singularity (type) is that of the internal function $f_2$ and the expansion is obtained by combining the regular expansion of $f_3$ with the singular expansion of $f_2$ at $R_{f_2}$, see \cite{Flajolet08}, section VI.9, p.\ 411. We obtain
\begin{align}
Q(z)={} & (1-\rho)\exp\left(f_2(R_Q)-\dfrac{\lambda}{\nu}f_1(R_Q)(R_Q-z)+\dfrac{2\lambda}{3\nu}c^{(2)}_{f_1}(R_Q-z)^{3/2}+O((R_Q-z)^2)\right) \nonumber \\
\sim{} & Q(R_Q)-(1-\rho)\exp\left(f_2(R_Q)\right)\dfrac{\lambda}{\nu}f_1(R_Q)(R_Q-z)+(1-\rho)\exp\left(f_2(R_Q)\right)\dfrac{2\lambda}{3\nu}c^{(2)}_{f_1}(R_Q-z)^{3/2}, \nonumber
\end{align}
as $z\rightarrow R_Q$. 
\item $R_Q=R^*=R_h$: $f_2(R_{f_2})<\infty$ and this case is also of the subcritical type. We have $R_Q=R_{f_2}$ and
\begin{align}
Q(z)={} & (1-\rho)\exp\left(f_2(R_Q)\right)\exp\left(-\dfrac{2\lambda}{\nu}c^{(3)}_{f_1}(R_Q-z)^{1/2}+O(R_Q-z)\right) \nonumber \\
\sim{} & Q(R_Q)-(1-\rho)\exp\left(f_2(R_Q)\right)\dfrac{2\lambda}{\nu}c^{(3)}_{f_1}(R_Q-z)^{1/2}, \nonumber
\end{align}
as $z\rightarrow R_Q$.
\end{enumerate}
\end{proof}

\begin{corollary}
Define $q(n)$ as the probability that $n$ customers are in orbit and the server is free. Then
\begin{align}
q(n)\sim{} & \left\{\begin{array}{cc}
\dfrac{c_Q^{(1)}}{\Gamma\left(\dfrac{\lambda}{\nu}c^{(1)}_{f_1}\right)}n^{\dfrac{\lambda}{\nu}c^{(1)}_{f_1}-1}R_Q^{-n-\dfrac{\lambda}{\nu}c^{(1)}_{f_1}} & \mbox{ if } R_Q=R^*<R_h \\
\dfrac{3c_Q^{(2)}}{4\sqrt{\pi}}n^{-5/2}R_Q^{-n+3/2} & \mbox{ if } R_Q=R_h \mbox{ and } R^* \mbox{ does not exist} \\
\dfrac{c_Q^{(3)}}{2\sqrt{\pi}}n^{-3/2}R_Q^{-n+1/2} & \mbox{ if } R_Q=R^*=R_h \\
\end{array}\right.. \nonumber
\end{align}
\label{corollary_asymptotics_q}
\end{corollary}
\begin{proof}
We have that $q(n)=[z^n]Q(z)$. The corollary follows directly from the asymptotic behavior of function $Q(z)$ derived in Theorem \ref{theorem_asymptotics_Q}, the general inversion formula \eqref{asymptotic_inversion} and two properties of the Gamma-function, namely $\Gamma(n+1)=n\Gamma(n)$ and $\Gamma(1/2)=\sqrt{\pi}$.
\end{proof}

\subsection{Asymptotic Inversion of $R(z)$}

The asymptotic behavior of $R(z)$ and its asymptotic inversion are closely related to those of $Q(z)$, since $Q(z)$ is a factor of $R(z)$, cf.\ expressions \eqref{Q} and \eqref{R}. We have following theorem and corollary.

\begin{theorem}
The dominant singularity $R_R$ of $R(z)$ is equal to $R_Q$. The function $R(z)$ behaves asymptotically as
\begin{align}
R(z)\sim{} & \left\{\begin{array}{cc}
\dfrac{\lambda}{\lambda_2}c_Q^{(1)}c^{(1)}_{f_1}(R_R-z)^{-1-\dfrac{\lambda}{\nu}c^{(1)}_{f_1}} & \mbox{ if } R_R=R^*<R_h \\
R(R_R)-\dfrac{\lambda}{\lambda_2}c^{(2)}_{f_1}Q(R_R)(R_R-z)^{1/2} & \mbox{ if } R_R=R_h \mbox{ and } R^* \mbox{ does not exist} \\
\dfrac{\lambda}{\lambda_2}\dfrac{c^{(3)}_{f_1}Q(R_R)}{(R_R-z)^{1/2}} & \mbox{ if } R_R=R^*=R_h \\
\end{array}\right., \label{asymptotics_R}
\end{align} 
as $z\rightarrow R_R$.
\label{theorem_asymptotics_R}
\end{theorem}

\begin{proof}
The function $R(z)$ can be written as
\begin{align}
R(z)={} & \dfrac{\lambda}{\lambda_2}f_1(z)Q(z). \nonumber
\end{align}
Both $f_1(z)$ and $Q(z)$ have the same dominant singularity, namely $R_h$ or $R^*$, see Lemma \ref{lemma_f_1} and Theorem \ref{theorem_asymptotics_Q}. Multiplying both asymptotic expressions and keeping the dominant terms lead to the theorem.
\end{proof}

\begin{corollary}
Define $q(n)$ as the probability that $n$ customers are in orbit and the server is busy. Then
\begin{align}
r(n)\sim{} & \left\{\begin{array}{cc}
\dfrac{\lambda}{\lambda_2}\dfrac{c_Q^{(1)}c^{(1)}_{f_1}}{\Gamma\left(1+\dfrac{\lambda}{\nu}c^{(1)}_{f_1}\right)}n^{\dfrac{\lambda}{\nu}c^{(1)}_{f_1}}R_R^{-n-1-\dfrac{\lambda}{\nu}c^{(1)}_{f_1}} & \mbox{ if } R_R=R^*<R_h \\
\dfrac{\lambda}{\lambda_2}\dfrac{c^{(2)}_{f_1}Q(R_R)}{2\sqrt{\pi}}n^{-3/2}R_R^{-n+1/2} & \mbox{ if } R_R=R_h \mbox{ and } R^* \mbox{ does not exist} \\
\dfrac{\lambda}{\lambda_2}\dfrac{c^{(3)}_{f_1}Q(R_R)}{\sqrt{\pi}}n^{-1/2}R_R^{-n-1/2} & \mbox{ if } R_R=R^*=R_h \\
\end{array}\right.. \nonumber
\end{align}
\label{corollary_asymptotics_r}
\end{corollary}
\begin{proof}
We have that $r(n)=[z^n]R(z)$. The corollary follows directly from the asymptotic behavior of function $R(z)$ derived in Theorem \ref{theorem_asymptotics_R}, the general inversion formula \eqref{asymptotic_inversion} and two properties of the Gamma-function, namely $\Gamma(n+1)=n\Gamma(n)$ and $\Gamma(1/2)=\sqrt{\pi}$.
\end{proof}

\subsection{Asymptotic Inversion of $P_2(z)$}

Finally, we calculate the asymptotics of the distribution $p_2(n)$ of the number of type-II customers in orbit. 
\begin{theorem}
The dominant singularity $R_{P_2}$ of $P_2(z)$ is equal to $R_R$. Furthermore, 
\begin{align}
P_2(z)\sim & \left\{\begin{array}{cc}
\dfrac{\lambda}{\lambda_2}c_Q^{(1)}c^{(1)}_{f_1}(R_{P_2}-z)^{-1-\dfrac{\lambda}{\nu}c^{(1)}_{f_1}} & \mbox{ if } R_{P_2}=R^*<R_h \\
P_2(R_{P_2})-\dfrac{\lambda}{\lambda_2}c^{(2)}_{f_1}Q(R_{P_2})(R_{P_2}-z)^{1/2} & \mbox{ if } R_{P_2}=R_h \mbox{ and } R^* \mbox{ does not exist} \\
\dfrac{\lambda}{\lambda_2}\dfrac{c^{(3)}_{f_1}Q(R_{P_2})}{(R_{P_2}-z)^{1/2}} & \mbox{ if } R_{P_2}=R^*=R_h \\
\end{array}\right., \nonumber
\end{align}
as $z\rightarrow R_{P_2}$.
\label{theorem_asymptotics_P2}
\end{theorem}
\begin{proof}
From $P_2(z)=Q(z)+R(z)$, see \eqref{P2}, and $R_Q=R_R$, it follows that $R_{P_2}=R_Q=R_R$. The asymptotic behavior of $P_2(z)$ is the asymptotic behavior of the sum of formulas \eqref{asymptotics_Q} and \eqref{asymptotics_R}. However, the asymptotics of $Q(z)$ are negligible to the asymptotics of $R(z)$ except for the constant part in the second case. The theorem follows directly.
\end{proof}
\begin{corollary}
The asymptotics of the distribution $p_2(n)$ of the number of type-II customers in orbit is given by
\begin{align}
p_2(n)\sim{} & \left\{\begin{array}{cc}
\dfrac{\lambda}{\lambda_2}\dfrac{c_Q^{(1)}c^{(1)}_{f_1}}{\Gamma\left(1+\dfrac{\lambda}{\nu}c^{(1)}_{f_1}\right)}n^{\dfrac{\lambda}{\nu}c^{(1)}_{f_1}}R_{P_2}^{-n-1-\dfrac{\lambda}{\nu}c^{(1)}_{f_1}} & \mbox{ if } R_{P_2}=R^*<R_h \\
\dfrac{\lambda}{\lambda_2}\dfrac{c^{(2)}_{f_1}Q(R_{P_2})}{2\sqrt{\pi}}n^{-3/2}R_{P_2}^{-n+1/2} & \mbox{ if } R_{P_2}=R_h \mbox{ and } R^* \mbox{ does not exist} \\
\dfrac{\lambda}{\lambda_2}\dfrac{c^{(3)}_{f_1}Q(R_{P_2})}{\sqrt{\pi}}n^{-1/2}R_{P_2}^{-n-1/2} & \mbox{ if } R_{P_2}=R^*=R_h \\
\end{array}\right.. \label{asymptotics_p2}
\end{align}
\label{corollary_asymptotics_p2}
\end{corollary}

\begin{proof}
Theorems \ref{theorem_asymptotics_R} and \ref{theorem_asymptotics_P2} show that the asymptotic behavior of $R(z)$ and $P_2(z)$ are equal, except for the constant in the second case. Since this constant has no influence on the asymptotic distribution, we find $p_2(n)\sim r(n)$.
\end{proof}

\subsection{Special Cases}

Before discussing the derived asymptotics generally, we write down the expressions of the asymptotics for some special cases.

First, the retrial queue without regular buffer space is obtained for $\lambda_1=0$. We find following corollary.

\begin{corollary}
For the classical retrial queue without regular buffer space, the asymptotics of the number of customers in orbit when the server is free, when the server is busy and unconditionally is respectively given by
\begin{align}
q(n)\sim{} & \dfrac{c_Q^{(1)}}{\Gamma\left(\dfrac{\lambda_2}{\nu}c^{(1)}_{f_1}\right)}n^{\dfrac{\lambda_2}{\nu}c^{(1)}_{f_1}-1}R_Q^{-n-\dfrac{\lambda_2}{\nu}c^{(1)}_{f_1}}, \nonumber \\
r(n)\sim{} & \dfrac{c_Q^{(1)}c^{(1)}_{f_1}}{\Gamma\left(1+\dfrac{\lambda_2}{\nu}c^{(1)}_{f_1}\right)}n^{\dfrac{\lambda_2}{\nu}c^{(1)}_{f_1}}R_R^{-n-1-\dfrac{\lambda_2}{\nu}c^{(1)}_{f_1}}, \nonumber \\
p_2(n)\sim{} & \dfrac{c_Q^{(1)}c^{(1)}_{f_1}}{\Gamma\left(1+\dfrac{\lambda_2}{\nu}c^{(1)}_{f_1}\right)}n^{\dfrac{\lambda_2}{\nu}c^{(1)}_{f_1}}R_{P_2}^{-n-1-\dfrac{\lambda_2}{\nu}c^{(1)}_{f_1}}, \nonumber
\end{align}
as $n\rightarrow\infty$. Here, $\rho_2=\lambda_2$E$[b]$, $R_Q=R_R=R_{P_2}=R^*$, $R^*$ is the smallest positive zero larger than $1$ of $B^*(\lambda_2(1-z))-z$ and
\begin{align}
c^{(1)}_{f_1}={} &  \dfrac{1-R^*}{1+\lambda_2{B^*}'(\lambda_2(1-R^*))}, \nonumber \\
c_Q^{(1)}={} & (1-\rho_2)(R_Q-1)^{\dfrac{\lambda_2}{\nu}c^{(1)}_{f_1}}\exp\left(-\dfrac{\lambda_2}{\nu}\int_{R_Q}^1\left(\dfrac{1-B^*(\lambda_2(1-z))}{B^*(\lambda_2(1-a))-x}-\dfrac{c^{(1)}_{f_1}}{R_Q-x}\right)dx\right). \nonumber
\end{align} 
\end{corollary}

\begin{proof}
For $\lambda_1=0$, $h(z)=B^*(\lambda_2(1-z))$ and is no longer implicitly defined. Therefore, the regular pole $R^*$ of $f_1(z)$ is a zero of $B^*(\lambda_2(1-z))-z$, exists always and is dominant. The expressions in the corollary follow by substituting $\lambda_1=0$ in the first cases of the respective expressions in corollaries \ref{corollary_asymptotics_q}, \ref{corollary_asymptotics_r} and \ref{corollary_asymptotics_p2}. 
\end{proof}
This corollary is consistent with the results of \cite{Kim_Ko07}.

Secondly, the non-retrial priority queue is obtained when the retrial rate goes to infinity. Note that $Q(z)=1-\rho$ in this case, since the server is free only when the system is empty. This also follows from letting $\nu\rightarrow\infty$ in expression \eqref{Q}. We therefore only look into the asymptotic distribution of the number of class-II customers (unconditionally).
\begin{corollary}
Assume the regular two-class $M/G/1$ priority queue. The asymptotics of the number of class-II (low-priority) customers in queue are given by
\begin{align}
p_2(n)\sim{} & \left\{\begin{array}{cc}
\dfrac{\lambda}{\lambda_2}(1-\rho)c^{(1)}_{f_1}R_{P_2}^{-n-1} & \mbox{ if } R_{P_2}=R^*<R_h \\
\dfrac{\lambda}{\lambda_2}(1-\rho)\dfrac{c^{(2)}_{f_1}}{2\sqrt{\pi}}n^{-3/2}R_{P_2}^{-n+1/2} & \mbox{ if } R_{P_2}=R_h \mbox{ and } R^* \mbox{ does not exist} \\
\dfrac{\lambda}{\lambda_2}(1-\rho)\dfrac{c^{(3)}_{f_1}}{\sqrt{\pi}}n^{-1/2}R_{P_2}^{-n-1/2} & \mbox{ if } R_{P_2}=R^*=R_h \\
\end{array}\right.. \nonumber
\end{align}
as $n\rightarrow\infty$, with $c_{f_1}^{(i)}$ as in Lemma \ref{lemma_f_1}.
\label{corollary_asymptotics_nu=infinity}
\end{corollary}

\begin{proof}
This lemma is obtained by letting $\nu\rightarrow\infty$ in formula \eqref{asymptotics_p2}, by using that the PGF $P_2(z)$ converges to the correct PGF of the number of class-II customers in queue for $\nu\rightarrow\infty$ and Continuity Theorem IX.1, p.\ 624 of \cite{Flajolet08}.
\end{proof}

\subsection{Observations and Discussion}

We end this section with a few observations and some discussion on the results.

\begin{itemize}
\item All obtained asymptotic distributions are power law distributions with exponential cut-off. 
\item The exponential factors of the three distributions are equal and completely determined by the dominant singularity of $f_1(z)$. More precisely, they are equal to the inverse of that dominant singularity. The dominant singularity depends on the distribution of the service times and of the arrival rates of both types of traffic. It does however not depend on the retrial rate $\nu$. As a consequence, the distribution of the number of class-II customers in the non-retrial priority queue has the exact same exponential factor as that in the retrial priority queue (see corollaries \ref{corollary_asymptotics_p2} and \ref{corollary_asymptotics_nu=infinity}).
\item The sub-exponential (power-law) factor is more involved.
\begin{itemize}
\item First, the scaling exponent of the power-law factor of $q(n)$ is one unit smaller than that of $r(n)$ and $p_2(n)$. This means that the non-exponential part of the asymptotics of the distribution of the number of type-II customers is heavier unconditionally or when the server is busy than when the server is free. This makes sense since the number of customers in orbit can only decrease in case the system is free of class-I customers. 
\item Second, for each of the distributions (let us concentrate on $p_2(n)$) different scaling exponents are possible. It is either equal to $(\lambda/\nu)c_{f_1}^{(1)}$, $-3/2$ or $-1/2$. The first depends on all parameters of the queueing system and the distribution of the service times. The other two are constant.
\item Which sub-exponential factor of the three possibilities applies depends on the dominant singularity of $f_1(z)$, just as is the case for the exponential factor. When the regular single pole $R^*$ is dominant, the scaling exponent equals $(\lambda/\nu)c_{f_1}^{(1)}$. When the square-root singularity of $h(z)$ dominates, it equals $-3/2$. When both singularities coincide, i.e., in case the square-root singularity equals the pole that makes the denominator of $f_1(z)$ equal to zero, the scaling exponent is $-1/2$. 
\item Which of the three cases applies is a function of all parameters and the distribution of the service times, except for the retrial rate $\nu$. In case $\nu=\infty$, the queueing model becomes a regular (non-retrial) priority queue. The invariance under change of $\nu$ means that the division in the three classes is identical for a retrial priority queue and a non-retrial priority queue, or in other words, if the first case applies for the non-retrial queue it also applies for the retrial queue. The regular priority case has been studied before. In the specific case of \textit{exponential} service times, which one of the three cases applies depends on the sign of
\begin{align}
D={} & \lambda_2\mu-(\lambda+\mu-2\sqrt{\lambda_1\mu})\sqrt{\lambda_1\mu}, \nonumber
\end{align}
see formula (3.1) in \cite{Li_Zhao11b}. We show the splitting of the parameterspace $(\lambda_1,\lambda_2)$ by the curve for $D=0$ in Figure \ref{tailspace} for $\mu=1$. When $D>0$ the scaling exponent equals $(\lambda/\nu)c_{f_1}^{(1)}$ ($0$ in the non-retrial case); when $D<0$ the scaling exponent is equal to $-3/2$; when $D=0$, it is equal to $-1/2$. This condition is explicit in the parameters $\lambda_1$, $\lambda_2$ and $\mu$. The relation of $D$ in terms of $\lambda_1$ is not necessarily monotonic; in terms of $\lambda_2$, however, it is monotonously increasing, since $\partial D/\partial \lambda_2=\mu-\sqrt{\lambda_1\mu}$ and positiveness of this derivative follows from the stability condition. For small $\lambda_2$, $D<0$ and the asymptotic distribution of the number of class-II customers in orbit is characterized by $R_h$, the dominant singularity of $h(z)$. We call this the \textit{priority regime}, as large type-II number of customers in the orbits are primarily caused by the priority of class-I customers over these class-II customers. In fact, in this case, the asymptotics are closely related to the asymptotics of the busy period of the high-priority customers, cf.\ \cite{Abate97}. For larger $\lambda_2$, $D$ increases and becomes positive. In this case, the asymptotic distribution of the number of class-II customers in orbit is characterized by $R^*$, the regular pole of $f_1(z)$. In case of a non-retrial priority queue ($\nu=\infty$), this leads to exponential asymptotics of the type-II customers; for the general retrial priority queue, an extra power law factor emerges, with scaling exponent depending on the retrial rate $\nu$. In this case, large numbers of class-II customers are primarily caused by the shear number of arriving type-II customers. Therefore, we call this the \textit{retrial regime}. However, we must note that the scaling exponent in this case also depends on the high-priority arrival rate, which means that the priority scheduling also plays a role in (the type of) the asymptotics. The third case is the boundary case of the former two cases. We note however that the scaling exponent in this boundary case is a constant ($-1/2$) and independent of the retrial rate $\nu$.
\begin{figure}
\centering
\includegraphics[width=0.5\textwidth]{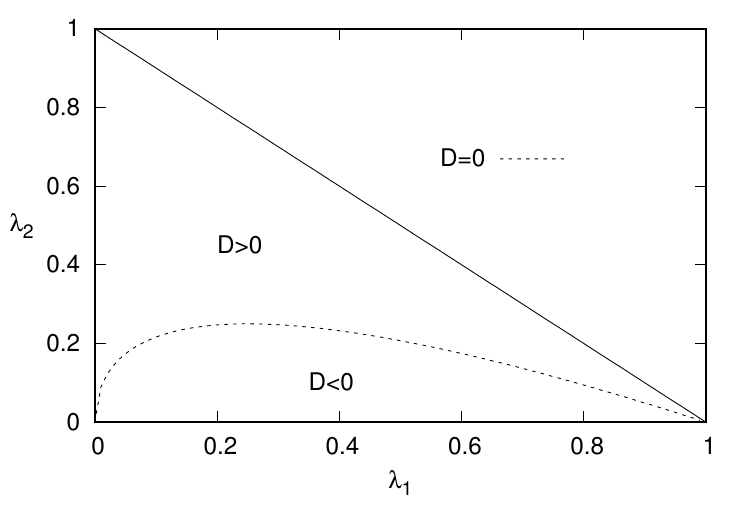}
\caption{The partitioning of the tailspace $(\lambda_1,\lambda_2)$ depending on the sign of $D$ for exponential service times with service rate $1$. The line $\lambda_1+\lambda_2=1$ indicates the stability border.}
\label{tailspace}
\end{figure}
\end{itemize}
\end{itemize}

\section{Type-2 and Type-3 Service Time Distributions} \label{power_law}

We now discuss an important subclass of service-time distributions of type 2 or type 3. For type 2 and type 3 service time distributions, the function value of the LST in its dominant singularity $-R_{B^*}$ is finite. The subclass we study is the one where the LST of the service time distribution is asymptotically equal to
\begin{align}
B^*(s)\sim{} & \sum_{j=0}^{\lfloor -\alpha_{B^*}\rfloor} \dfrac{B^{*^{(j)}}(-R_{B^*})}{j!}(s+R_{B^*})^j+c_{B^*}(s+R_{B^*})^{-\alpha_{B^*}}, \label{expansion_Bstar}
\end{align}
as $-s\rightarrow R_{B^*}$, with $R_{B^*}\geq0$ and $\alpha_{B^*}<0$ and not an integer.\footnote{We write $-s\rightarrow R_{B^*}$ instead of $s\rightarrow -R_{B^*}$ such that the singularity is approached from the left, cf.\ definition of $\sim$ in section \ref{approach}.}\footnote{The subclass can be extended somewhat. For instance, the RHS of \eqref{expansion_Bstar} can be multiplied with a slowly varying function that is merely transferred to an extra (slowly varying) factor in the asymptotics of the distributions of interest, see \cite{Flajolet08}, remark VI.5., p.\ 386.} The two types of service time distributions correspond with $R_{B^*}>0$ (type 2) and $R_{B^*}=0$ (type 3). The first leads to a power-law service time distribution with exponential cut-off and is still light-tailed. The second leads to a `pure' power-law service time distribution. We analyze both cases separately in the remainder, but first discuss the main difference with the asymptotics in case of type-1 service time distributions.

When $B^*(-R_{B^*})$ is finite (type-2 and type-3), the square-root branch-point of $h(z)$ (cf.\ Lemma \ref{lemma_h}) and the regular pole of $f_1(z)$ (cf.\ Lemma \ref{lemma_f_1}) do not necessarily exist. Instead, the singularity that is introduced by the singularity $-R_B^*$ of $B^*(s)$ can be dominant in all functions, while this singularity is never dominant in case of type-1 distributions ($R_h$ always exists in that case, cf.\ the discussion in the Appendix). Loosely speaking, when $z$ is increased, $\lambda_1(1-h(z))+\lambda_2(1-z)$ equals $-R_{B^*}$ (causing $B^*(\lambda_1(1-h(z))+\lambda_2(1-z))$ to reach its dominant singularity) \textit{before} the singularity $R_h$ of $h(z)$ or the pole $R^*$ of $f_1(z)$ is reached. In fact, in case of type-3 distributions, this singularity equals $1$ and is therefore always dominant (PGFs are analytic for $|z|<1$). 

\subsection{Type-2 Service Time Distributions}

Starting with the function $f_1(z)$, three different singularities can be dominant in this case: $R_h$, $R^*$ or $R_{h^*}$, the singularity of $h(z)$ that originates from the singularity $-R_B^*$ of $B^*(s)$. If one of the first two is dominant, the analysis of the previous section applies. Therefore, in this section, we only discuss the case that $R_{h^*}$ is the dominant singularity of $h(z)$ and $f_1(z)$.

We start with calculation of $R_{h^*}$.

\begin{lemma}
In case the service time distribution is of type 2, the dominant singularity $-R_{B^*}$ of $B^*(s)$ introduces a singularity $R_{h^*}$ in the function $h(z)$ given by
\begin{align}
R_{h^*}=1+\dfrac{\lambda_1(1+R_{B^*}-B^*(R_{B^*}))}{\lambda_2}, \nonumber
\end{align}
if $R_{h^*}\leq R_h$.
\label{lemma_Rhstar}
\end{lemma}

\begin{proof}
Since $h(z)$ is implicitly defined as $h(z)=B^*(\lambda_1(1-h(z))+\lambda_2(1-z))$, we find that $R_{h^*}$ that satisfies the set of equations 
\begin{align}
\lambda_1(1-h(R_{h^*}))+\lambda_2(1-R_{h^*})={} & -R_{B^*}, \label{-R_Bstar} \\
h(R_{h^*})={} & B^*(-R_{B^*}), \label{Bstar(-R_Bstar)} 
\end{align}
is a singularity of $h(z)$, if it is smaller than $R_h$ ($h(z)$ is not defined for $z>R_h$ if not). This set of equations can be explicitly solved and results in
\begin{align}
R_{h^*}=1+\dfrac{R_{B^*}+\lambda_1(1-B^*(R_{B^*}))}{\lambda_2}. \label{Rhstar}
\end{align}
\end{proof}

Next, we  identify the dominant singularities of $Q(z)$, $R(z)$ and $P_2(z)$ and analyze their behavior in the neighborhood of these singularities.

\begin{theorem}
When $R_{h^*}$ as defined in Lemma \ref{lemma_Rhstar} is the dominant singularity of $f_1(z)$, we find that the dominant singularities of $Q(z)$, $R(z)$ and $P_2(z)$ are all equal to $R_{h^*}$, and they asymptotically behave like
\begin{align}
Q(z)\sim{} & \sum_{j=0}^{\lfloor-\alpha_{B^*}\rfloor+1}\dfrac{(-1)^jQ^{(j)}(R_Q)}{j!}(R_Q-z)^{j} \nonumber \\
& -c_{B^*}\dfrac{\lambda}{\nu}(1-\rho)\exp\left(f_2(R_Q)\right)\dfrac{(\lambda_1h'(R_Q)+\lambda_2)^{-\alpha_{B^*}}(R_Q-1)}{(-\alpha_{B^*}+1)(h(R_Q)-R_Q)^2}(R_Q-z)^{-\alpha_{B^*}+1}, \nonumber \\
R(z)\sim{} & \sum_{j=0}^{\lfloor-\alpha_{B^*}\rfloor}\dfrac{(-1)^jR^{(j)}(R_R)}{j!}(R_R-z)^{j}+c_{B^*}\dfrac{\lambda}{\lambda_2}\dfrac{(\lambda_1h'(R_R)+\lambda_2)^{-\alpha_{B^*}}(R_R-1)}{(h(R_R)-R_R)^2}Q(R_R)(R_R-z)^{-\alpha_{B^*}}, \nonumber \\
P_2(z)\sim & \sum_{j=0}^{\lfloor-\alpha_{B^*}\rfloor}\dfrac{(-1)^jP_2^{(j)}(R_{P_2})}{j!}(R_{P_2}-z)^{j}+c_{B^*}\dfrac{\lambda}{\lambda_2}\dfrac{(\lambda_1h'(R_{P_2})+\lambda_2)^{-\alpha_{B^*}}(R_{P_2}-1)}{(h(R_{P_2})-R_{P_2})^2}Q(R_{P_2})(R_{P_2}-z)^{-\alpha_{B^*}}, \nonumber
\end{align}
as $z\rightarrow R_Q=R_R=R_{P_2}=R_{h^*}$.
\label{theorem_type2}
\end{theorem}

\begin{proof}
We first study the behavior of $h(z)$ in the neighborhood of $R_{h^*}$. By substituting $s$ by $\lambda_1(1-h(z))+\lambda_2(1-z)$ in the expansion \eqref{expansion_Bstar} of $B^*(s)$ about $-R_{B^*}$, we can write
\begin{align}
h(z)={} & B^*(\lambda_1(1-h(z))+\lambda_2(1-z)) \nonumber \\
\sim{} & \sum_{j=0}^{\lfloor -\alpha_{B^*}\rfloor} \dfrac{B^{*^{(j)}}(-R_{B^*})}{j!}(\lambda_1(1-h(z))+\lambda_2(1-z)+R_{B^*})^j+c_{B^*}(\lambda_1(1-h(z))+\lambda_2(1-z)+R_{B^*})^{-\alpha_{B^*}} \nonumber
\end{align}
as $-\lambda_1(1-h(z))-\lambda_2(1-z)\rightarrow R_{B^*}$. Remark that the latter makes sense since $h(z)$ is an increasing function and therefore $-\lambda_1(1-h(z))-\lambda_2(1-z)$ is an increasing function as well. Furthermore, from the proof of Lemma \ref{lemma_Rhstar}, we know that $\lambda_1(1-h(R_{h^*}))+\lambda_2 (1-R_{h^*})=-R_{B^*}$, i.e., $-\lambda_1(1-h(z))+\lambda_2(1-z)\rightarrow R_{B^*}$ if $z\rightarrow R_{h^*}$. By writing $R_{B^*}$ as a function of $R_{h^*}$ and $h(R_{h^*})$, cf.\ equation \eqref{-R_Bstar}, we can write
\begin{align}
h(z)\sim{} & \sum_{j=0}^{\lfloor -\alpha_{B^*}\rfloor} \dfrac{B^{*^{(j)}}(-R_{B^*})}{j!}(\lambda_1(h(R_{h^*})-h(z))+\lambda_2(R_{h^*}-z))^j+c_{B^*}(\lambda_1(h(R_{h^*})-h(z))+\lambda_2(R_{h^*}-z))^{-\alpha_{B^*}}, \nonumber
\end{align}
or
\begin{align}
h(&R_{h^*})-h(z) \nonumber \\
& \sim{} -\sum_{j=1}^{\lfloor -\alpha_{B^*}\rfloor} \dfrac{B^{*^{(j)}}(-R_{B^*})}{j!}(\lambda_1(h(R_{h^*})-h(z))+\lambda_2(R_{h^*}-z))^j-c_{B^*}(\lambda_1(h(R_{h^*})-h(z))+\lambda_2(R_{h^*}-z))^{-\alpha_{B^*}}, \label{h(Rhstar)-h(z)}
\end{align}
as $z\rightarrow R_{h^*}$, where we used \eqref{Bstar(-R_Bstar)}. This equation can be used recursively to eliminate the powers of $h(R_{h^*})-h(z)$ in the RHS. However, the first terms will then evolve to the Taylor expansion of $h(z)$ about $h(R_{h^*})$. We can thus write
\begin{align}
h(R_{h^*})-h(z)\sim{} & -\sum_{j=1}^{\lfloor -\alpha_{B^*}\rfloor} \dfrac{(-1)^jh^{(j)}(R_{h^*})}{j!}(R_{h^*}-z)^j-c_{B^*}(\lambda_1(h(R_{h^*})-h(z))+\lambda_2(R_{h^*}-z))^{-\alpha_{B^*}}, \nonumber
\end{align}
as $z\rightarrow R_{h^*}$. Finally, to eliminate the factor $h(R_{h^*})-h(z)$ in the last term, we can use a less informative version of this equation:
\begin{align}
h(R_{h^*})-h(z)\sim{} & h'(R_{h^*})(R_{h^*}-z), \nonumber
\end{align}
as $z\rightarrow R_{h^*}$. By substituting this in the RHS of equation \eqref{h(Rhstar)-h(z)}, we find
\begin{align}
h(R_{h^*})-h(z)\sim{} & -\sum_{j=1}^{\lfloor -\alpha_{B^*}\rfloor} \dfrac{(-1)^jh^{(j)}(R_{h^*})}{j!}(R_{h^*}-z)^j-c_{B^*}(\lambda_1h'(R_{h^*})+\lambda_2)^{-\alpha_{B^*}}(R_{h^*}-z)^{-\alpha_{B^*}}, \nonumber
\end{align}
and
\begin{align}
h(z)\sim{} & \sum_{j=0}^{\lfloor -\alpha_{B^*}\rfloor} \dfrac{(-1)^jh^{(j)}(R_{h^*})}{j!}(R_{h^*}-z)^j+c_{B^*}(\lambda_1h'(R_{h^*})+\lambda_2)^{-\alpha_{B^*}}(R_{h^*}-z)^{-\alpha_{B^*}}, \label{asymptotics_h_type2}
\end{align}
as $z\rightarrow R_{h^*}$.

Next, we look at asymptotics of $f_1(z)$ in the neighborhood of $R_{f_1}=R_{h^*}$. By substituting (\ref{asymptotics_h_type2}) in (\ref{f1}), we can write
\begin{align}
f_1(z)\sim{} & \dfrac{1-h(R_{f_1})-\sum_{j=1}^{\lfloor-\alpha_{B^*}\rfloor} \dfrac{(-1)^jh^{(j)}(R_{f_1})}{j!}(R_{f_1}-z)^j-c_{B^*}(\lambda_1h'(R_{f_1})+\lambda_2)^{-\alpha_{B^*}}(R_{f_1}-z)^{-\alpha_{B^*}}}{h(R_{f_1})-R_{f_1}+\sum_{j=1}^{\lfloor -\alpha_{B^*}\rfloor} \dfrac{(-1)^jh^{(j)}(R_{f_1})}{j!}(R_{f_1}-z)^j+c_{B^*}(\lambda_1h'(R_{f_1})+\lambda_2)^{-\alpha_{B^*}}(R_{f_1}-z)^{-\alpha_{B^*}}}, \nonumber \\
\sim{} & \sum_{j=0}^{\lfloor-\alpha_{B^*}\rfloor}\dfrac{(-1)^jf_1^{(j)}(R_{f_1})}{j!}(R_{f_1}-z)^j+c_{B^*}\dfrac{(\lambda_1h'(R_{f_1})+\lambda_2)^{-\alpha_{B^*}}(R_{f_1}-1)}{(h(R_{f_1})-R_{f_1})^2}(R_{f_1}-z)^{-\alpha_{B^*}}, \label{asymptotics_f1_type2} 
\end{align}
as $z\rightarrow R_{f_1}$. This is found by expanding the denominator about $R_{f_1}$, multiplying it with the numerator and only keeping the dominant terms in the neighborhood of $R_{f_1}$.

The dominant singularity $R_{f_2}$ of $f_2(z)$ is equal to that of $f_1(z)$. We use \cite{Flajolet08}, Theorem VI.9, p.\ 420 to find
\begin{align}
f_2(z)\sim{} & \sum_{j=0}^{\lfloor-\alpha_{B^*}\rfloor+1}\dfrac{(-1)^jf_2^{(j)}(R_{f_2})}{j!}(R_{f_2}-z)^{j}-c_{B^*}\dfrac{\lambda}{\nu}\dfrac{(\lambda_1h'(R_{f_2})+\lambda_2)^{-\alpha_{B^*}}(R_{f_2}-1)}{(-\alpha_{B^*}+1)(h(R_{f_2})-R_{f_2})^2}(R_{f_2}-z)^{-\alpha_{B^*}+1}, \label{asymptotics_f2_type2}
\end{align}
as $z\rightarrow R_{f_2}$.

Finally, we can investigate the asymptotic behavior of $Q(z)$, $R(z)$ and $P_2(z)$. We start with $Q(z)$. Its dominant singularity $R_Q$ is equal to $R_{f_2}$ and by substituting \eqref{asymptotics_f2_type2} in $Q(z)=(1-\rho)\exp(f_2(z))$, we find
\begin{align}
Q(z)\sim{} & (1-\rho)\exp\left(f_2(R_Q)+\sum_{j=1}^{\lfloor-\alpha_{B^*}\rfloor+1}\dfrac{(-1)^jf_2^{(j)}(R_Q)}{j!}(R_Q-z)^{j}\right. \nonumber \\
& \left.-c_{B^*}\dfrac{\lambda}{\nu}\dfrac{(\lambda_1h'(R_Q)+\lambda_2)^{-\alpha_{B^*}}(R_Q-1)}{(-\alpha_{B^*}+1)(h(R_Q)-R_Q)^2}(R_Q-z)^{-\alpha_{B^*}+1}\right) \nonumber \\
\sim{} & (1-\rho)\exp\left(f_2(R_Q)\right)\exp\left(\sum_{j=1}^{\lfloor-\alpha_{B^*}\rfloor+1}\dfrac{(-1)^jf_2^{(j)}(R_Q)}{j!}(R_Q-z)^{j}\right) \nonumber \\
& \cdot\exp\left(-c_{B^*}\dfrac{\lambda}{\nu}\dfrac{(\lambda_1h'(R_Q)+\lambda_2)^{-\alpha_{B^*}}(R_Q-1)}{(-\alpha_{B^*}+1)(h(R_Q)-R_Q)^2}(R_Q-z)^{-\alpha_{B^*}+1}\right), \nonumber \\
\sim{} & \sum_{j=0}^{\lfloor-\alpha_{B^*}\rfloor+1}\dfrac{(-1)^jQ^{(j)}(R_Q)}{j!}(R_Q-z)^{j} \nonumber \\
& -c_{B^*}\dfrac{\lambda}{\nu}(1-\rho)\exp\left(f_2(R_Q)\right)\dfrac{(\lambda_1h'(R_Q)+\lambda_2)^{-\alpha_{B^*}}(R_Q-1)}{(-\alpha_{B^*}+1)(h(R_Q)-R_Q)^2}(R_Q-z)^{-\alpha_{B^*}+1}, \label{asymptotics_Q_type2}
\end{align}
as $z\rightarrow R_Q$. 

The dominant singularity $R_R$ of $R(z)$ is equal to $R_Q$. By substituting \eqref{asymptotics_f1_type2} and \eqref
{asymptotics_Q_type2} in $R(z)=\lambda/\lambda_2f_1(z)Q(z)$ and keeping the dominant terms, we find 
\begin{align}
R(z)\sim{} & \dfrac{\lambda}{\lambda_2}\left(\sum_{j=0}^{\lfloor-\alpha_{B^*}\rfloor}\dfrac{(-1)^jf_1^{(j)}(R_R)}{j!}(R_R-z)^j+c_{B^*}\dfrac{(\lambda_1h'(R_R)+\lambda_2)^{-\alpha_{B^*}}(R_R-1)}{(h(R_R)-R_R)^2}(R_R-z)^{-\alpha_{B^*}}\right) \nonumber \\
& \cdot\left(\sum_{j=0}^{\lfloor-\alpha_{B^*}\rfloor+1}\dfrac{(-1)^jQ^{(j)}(R_R)}{j!}(R_R-z)^{j}\right. \nonumber \\
& \left.-c_{B^*}\dfrac{\lambda}{\nu}(1-\rho)\exp\left(f_2(R_R)\right)\dfrac{(\lambda_1h'(R_R)+\lambda_2)^{-\alpha_{B^*}}(R_R-1)}{(-\alpha_{B^*}+1)(h(R_R)-R_R)^2}(R_R-z)^{-\alpha_{B^*}+1}\right) \nonumber \\
\sim{} & \sum_{j=0}^{\lfloor-\alpha_{B^*}\rfloor}\dfrac{(-1)^jR^{(j)}(R_R)}{j!}(R_R-z)^{j}+c_{B^*}\dfrac{\lambda}{\lambda_2}\dfrac{(\lambda_1h'(R_R)+\lambda_2)^{-\alpha_{B^*}}(R_R-1)}{(h(R_R)-R_R)^2}Q(R_R)(R_R-z)^{-\alpha_{B^*}}, \nonumber 
\end{align}
as $z\rightarrow R_R$.

Finally, since $P_2(z)=Q(z)+R(z)$, we have that $R_{P_2}=R_R=R_Q$ and
\begin{align}
P_2(z)\sim & \sum_{j=0}^{\lfloor-\alpha_{B^*}\rfloor}\dfrac{(-1)^jP_2^{(j)}(R_{P_2})}{j!}(R_{P_2}-z)^{j}+c_{B^*}\dfrac{\lambda}{\lambda_2}\dfrac{(\lambda_1h'(R_{P_2})+\lambda_2)^{-\alpha_{B^*}}(R_{P_2}-1)}{(h(R_{P_2})-R_{P_2})^2}Q(R_{P_2})(R_{P_2}-z)^{-\alpha_{B^*}}, \nonumber
\end{align}
as $z\rightarrow R_{P_2}$.
\end{proof}

\begin{corollary}
The asymptotics of the distributions $q(n)$, $r(n)$ and $p_2(n)$ are given by 
\begin{align}
q(n)\sim{} & c_{B^*}\dfrac{\lambda}{\nu}(1-\rho)\exp\left(f_2(R_Q)\right)\dfrac{(\lambda_1h'(R_Q)+\lambda_2)^{-\alpha_{B^*}})(R_Q-1)}{(h(R_Q)-R_Q)^2}\dfrac{n^{\alpha_{B^*}-2}}{\Gamma(\alpha_{B^*})}R_Q^{-n-\alpha_{B^*}+1}, \nonumber \\
r(n)\sim{} & c_{B^*}\dfrac{\lambda}{\lambda_2}\dfrac{(\lambda_1h'(R_R)+\lambda_2)^{-\alpha_{B^*}})(R_R-1)}{(h(R_R)-R_R)^2}Q(R_R)\dfrac{n^{\alpha_{B^*}-1}}{\Gamma(\alpha_{B^*})}R_R^{-n-\alpha_{B^*}}, \nonumber \\
p_2(n)\sim{} & c_{B^*}\dfrac{\lambda}{\lambda_2}\dfrac{(\lambda_1h'(R_{P_2})+\lambda_2)^{-\alpha_{B^*}})(R_{P_2}-1)}{(h(R_{P_2})-R_{P_2})^2}Q(R_{P_2})\dfrac{n^{\alpha_{B^*}-1}}{\Gamma(\alpha_{B^*})}R_{P_2}^{-n-\alpha_{B^*}}, \nonumber
\end{align}
as $n\rightarrow\infty$, with $R_Q=R_R=R_{P_2}=R_{h^*}$ and $R_{h^*}$ as given in \eqref{Rhstar}.
\end{corollary}

\begin{proof}
This follows directly from the asymptotic behavior of functions $Q(z)$, $R(z)$ and $P_2(z)$ derived in Theorem \ref{theorem_type2} and the general inversion formula \eqref{asymptotic_inversion}. In the expression of $q(n)$, we further used $\Gamma(\alpha_{B^*}-1)=\Gamma(\alpha_{B^*})/(\alpha_{B^*}-1)$.
\end{proof}

\subsubsection*{Remark}
We did not treat the case where one (or both) of the singularities $R_h$ and $R^*$ exist and coincide with $R_{h^*}$. Although this is possible and will lead to different asymptotics, we deem this as less important because these cases only apply to very particular combinations of the parameters. 

\subsection{Type-3 Service Time Distributions}
 
In case $R_{B^*}=0$, $0$ is a singularity of the LST $B^*(s)$ and the service times are asymptotically distributed according to a power law. This singularity is transferred to $1$ as dominant singularity of all PGFs. We have following theorem and corresponding corollary.

\begin{theorem}
In case the service time distribution is asymptotically power law, i.e.,
\begin{align}
B^*(s)\sim{} & \sum_{j=0}^{\lfloor -\alpha_{B^*}\rfloor} \dfrac{B^{*^{(j)}}(0)}{j!}s^j+c_{B^*}s^{-\alpha_{B^*}} \nonumber
\end{align}
as $-s\rightarrow0$, with $\alpha_{B^*}<-1$ and not an integer, the dominant singularity of $Q(z)$, $R(z)$ and $P_2(z)$ is equal to $1$ and they behave as
\begin{align}
Q(z)\sim{} & \sum_{j=0}^{\lfloor-\alpha_{B^*}\rfloor}\dfrac{(-1)^jQ^{(j)}(1)}{j!}(1-z)^{j}+c_{B^*}\dfrac{\lambda}{\nu}(1-\rho)\dfrac{(\lambda_1h'(1)+\lambda_2)^{-\alpha_{B^*}}}{-{\alpha_{B^*}}(1-h'(1))^2}(1-z)^{-\alpha_{B^*}}, \nonumber \\
R(z)\sim{} & \sum_{j=0}^{\lfloor-\alpha_{B^*}\rfloor-1}\dfrac{(-1)^jR^{(j)}(1)}{j!}(1-z)^{j}-c_{B^*}\dfrac{\lambda}{\lambda_2}\dfrac{(\lambda_1h'(1)+\lambda_2)^{-\alpha_{B^*}}}{(1-h'(1))^2}(1-z)^{-\alpha_{B^*}-1}, \nonumber \\
P_2(z)\sim{} & \sum_{j=0}^{\lfloor-\alpha_{B^*}\rfloor-1}\dfrac{(-1)^jP_2^{(j)}(1)}{j!}(1-z)^{j}-c_{B^*}\dfrac{\lambda}{\lambda_2}\dfrac{(\lambda_1h'(1)+\lambda_2)^{-\alpha_{B^*}}}{(1-h'(1))^2}(1-z)^{-\alpha_{B^*}-1}, \nonumber
\end{align}
as $z\rightarrow1$. 
\label{theorem_power}
\end{theorem}

\begin{proof}
We first study the behavior of $h(z)$ in the neighborhood of its dominant singularity. Lemma \ref{lemma_Rhstar} applies and leads to the dominant singularity $R_{h^*}=1$. In a similar way as in the first part of the proof of Theorem \ref{theorem_type2}, we find
\begin{align}
h(z)\sim{} & \sum_{j=0}^{\lfloor -\alpha_{B^*}\rfloor} \dfrac{(-1)^jh^{(j)}(1)}{j!}(1-z)^j+c_{B^*}(\lambda_1h'(1)+\lambda_2)^{-\alpha_{B^*}}(1-z)^{-\alpha_{B^*}}, \nonumber
\end{align}
as $z\rightarrow 1$.

Next, we look at the asymptotics of $f_1(z)=(1-h(z))/(h(z)-z)$ in the neighborhood of $1$. We can write
\begin{align}
f_1(z)\sim{} & \dfrac{-\sum_{j=1}^{\lfloor-\alpha_{B^*}\rfloor} \dfrac{(-1)^jh^{(j)}(1)}{j!}(1-z)^j-c_{B^*}(\lambda_1h'(1)+\lambda_2)^{-\alpha_{B^*}}(1-z)^{-\alpha_{B^*}}}{1-z+\sum_{j=1}^{\lfloor -\alpha_{B^*}\rfloor} \dfrac{(-1)^jh^{(j)}(1)}{j!}(1-z)^j+c_{B^*}(\lambda_1h'(1)+\lambda_2)^{-\alpha_{B^*}}(1-z)^{-\alpha_{B^*}}} \nonumber \\
={} & \dfrac{\sum_{j=0}^{\lfloor-\alpha_{B^*}\rfloor-1}\dfrac{(-1)^jh^{(j+1)}(1)}{(j+1)!}(1-z)^j-c_{B^*}(\lambda_1h'(1)+\lambda_2)^{-\alpha_{B^*}}(1-z)^{-\alpha_{B^*}-1}}{1-\sum_{j=0}^{\lfloor -\alpha_{B^*}\rfloor-1}\dfrac{(-1)^jh^{(j+1)}(1)}{(j+1)!}(1-z)^j+c_{B^*}(\lambda_1h'(1)+\lambda_2)^{-\alpha_{B^*}}(1-z)^{-\alpha_{B^*}-1}} \nonumber \\
\sim{} & \sum_{j=0}^{\lfloor-\alpha_{B^*}\rfloor-1}\dfrac{(-1)^jf_1^{(j)}(1)}{j!}(1-z)^j-c_{B^*}\dfrac{(\lambda_1h'(1)+\lambda_2)^{-\alpha_{B^*}}}{(1-h'(1))^2}(1-z)^{-\alpha_{B^*}-1}, \nonumber 
\end{align}
as $z\rightarrow1$. In the first step, we used that $h(1)=1$. In the second step, we canceled the common factor $(1-z)$ in numerator and denominator. In the last step, we expanded the denominator about $1$, multiplied it with the numerator and only kept the dominant terms in the neighborhood of $1$.

Next, we determine the behavior of $f_2(z)$ in the neighborhood of $1$. We have
\begin{align}
f_2(z)\sim{} & \sum_{j=0}^{\lfloor-\alpha_{B^*}\rfloor}\dfrac{(-1)^jf_2^{(j)}(1)}{j!}(1-z)^{j}+c_{B^*}\dfrac{\lambda}{\nu}\dfrac{(\lambda_1h'(1)+\lambda_2)^{-\alpha_{B^*}}}{-{\alpha_{B^*}}\cdot(1-h'(1))^2}(1-z)^{-\alpha_{B^*}}, \nonumber
\end{align}
as $z\rightarrow1$.

Finally, we can investigate the asymptotic behavior of $Q(z)$, $R(z)$ and $P_2(z)$ as $z\rightarrow1$. Since $f_2(1)=0$, we easily find the expressions in the theorem.
\end{proof}

From the asymptotic expressions of the PGFs in Theorem \ref{theorem_power}, we easily find the asymptotics of the distributions.

\begin{corollary}
In case of an asymptotically power-law distribution of the service times, the asymptotics of the distributions $q(n)$, $r(n)$ and $p_2(n)$ are given by 
\begin{align}
q(n)\sim{} & c_{B^*}\dfrac{\lambda}{\nu}(1-\rho)\dfrac{(\lambda_1h'(1)+\lambda_2)^{-\alpha_{B^*}}}{-{\alpha_{B^*}}\Gamma(\alpha_{B^*})(1-h'(1))^2}n^{\alpha_{B^*}-1}, \nonumber \\
r(n)\sim{} & c_{B^*}\dfrac{\lambda}{\lambda_2}\dfrac{(\lambda_1h'(1)+\lambda_2)^{-\alpha_{B^*}}}{-\alpha_{B^*}\Gamma(\alpha_{B^*})(1-h'(1))^2}n^{\alpha_{B^*}}, \nonumber \\
p_2(n)\sim{} & c_{B^*}\dfrac{\lambda}{\lambda_2}\dfrac{(\lambda_1h'(1)+\lambda_2)^{-\alpha_{B^*}}}{-\alpha_{B^*}\Gamma(\alpha_{B^*})(1-h'(1))^2}n^{\alpha_{B^*}}, \nonumber 
\end{align}
as $n\rightarrow\infty$.
\end{corollary}


\appendix

\section{The Smooth Implicit-Function Schema and the existence of $R_h$}

According to \cite{Flajolet08}, section VII.4.1, p.\ 467, the function $h(z)$ belongs to the smooth implicit-function schema if there exists a bivariate function $G(z,w)$ such that $h(z)=G(z,h(z))$ where $G(z,w)$ satisfies three conditions:
\begin{enumerate}
\item $G(z,w)=\sum_{m,n=0}^{\infty} g_{m,n} z^mw^n$ is analytic in a domain $|z|<R$ and $|w|<S$ for some $R,S>0$;
\item $g_{m,n}\geq0$, $g_{0,0}=0$, $g_{0,1}\neq1$ and $g_{m,n}>0$ for some $m$ and for some $n\geq2$;
\item There exists two numbers $r$ and $s$, such that $0<r<R$ and $0<s<S$, satisfying the system of equations
\begin{align}
G(r,s)={} & s, \nonumber \\
G_w(r,s)={} & 1. \nonumber 
\end{align}
\end{enumerate}

In our case, the first condition is satisfied for some $R,S\geq1$, since $G(z,w)=B^*(\lambda_1(1-w)+\lambda_2(1-z))$ is the joint PGF of the number of class-II and class-I arrivals in a service time. 

The second condition summarizes some technical aspects. The number $g_{m,n}$ is the probability that $m$ class-II and $n$ class-I customers arrive in a service time. Therefore, $g_{m,n}\geq0$. The number $g_{0,0}$ is clearly not 0, but this is merely a normalizing condition ($h(0)=0$ if it is satisfied). A more important condition is $g_{0,1}\neq1$, which is included to avoid that the implicit function be of reducible form $h=h+\ldots$ \cite{Flajolet08}, p.\ 468. This is clearly satisfied (if $g_{0,1}$ would be 1, all other probabilities would have to be 0 which is not the case). Furthermore, $g_{m,n}>0$ for some $m$ and $n\geq2$. In fact, because of the Poisson arrivals, $g_{m,n}>0$ for all $m$ and $n$. 

Finally, the $(r,s)$ in the third condition explicits the solution $(r_h,h(R_h))$ with $R_h$ the square-root singularity of $h(z)$ and $h(R_h)$ its function value. Let us first handle the case that the service time distribution is of type 1, i.e., $B^*(-R_{B^*})=\infty$. First $(1,1)$ is a solution of $G(z,w)=w$. Furthermore, $G_w(1,1)=\rho_1<1$. If we let $z$ increase, the point $w$ that satisfies $G(w,z)=w$ increases as well due to the implicit function theorem and due to $G_w(w,z)$ being an increasing function for $w,z\geq1$. Since $G_w(w,z)\rightarrow\infty$ eventually  when $\lambda_1(1-w)+\lambda_2(1-z)\rightarrow-R_{B^*}$, a couple $(r,s)$ must exist with $G(r,s)=s$ and $G_w(r,s)=1$. This reasoning also shows why such a solution does not necessarily exist for type-2 service time distributions. In that case, $G_w(w,z)$ can be smaller than $1$ for $w$ and $z$ that satisfy $G(w,z)=w$ and $\lambda_1(1-w)+\lambda_2(1-z)=-R_{B^*}$. In that case, $R_{h^*}$ is the dominant singularity of $h(z)$ and a solution $(r,s)$ of the set of equations $G(r,s)=s$, $G_w(r,s)=1$ does not exist. For type-3 service time  distributions, it is even guaranteed that such a solution does not exist. Here, $R_{h^*}=1$, $h(R_{h^*})=1$ and $G_w(1,1)=\rho_1<1$.


\subsubsection*{Acknowledgements}
The research of Tuan Phung-Duc is supported in part by University of Tsukuba Basic Research Support Program Type A.

%
%
%
%

\end{document}